\documentclass[12pt]{article}
\usepackage{amsmath,amssymb,amsthm,amsfonts,amsbsy, enumerate,graphicx,graphics,epstopdf}
\usepackage[dvips]{epsfig}

\theoremstyle{plain}
\newtheorem{theorem}{Theorem}[section]
\newtheorem{corollary}[theorem]{Corollary}

\newtheorem{lemma}[theorem]{Lemma}
\newtheorem{proposition}[theorem]{Proposition}

\newcommand{\ex}{{\rm ex}}
\newcommand{\F}{\mathbb{F}}

\begin{document}

\title{On the Spectrum of Wenger Graphs}
\author{Sebastian M. Cioab\u{a}, Felix Lazebnik  and Weiqiang Li \footnote{Department of Mathematical Sciences,  University of Delaware, Newark, DE 19707-2553, USA.}  }
\date{\today}
\maketitle

\begin{abstract}
Let $q=p^e$, where $p$ is a prime and $e\geq 1$ is an integer. For $m\geq 1$, let $P$ and $L$ be two copies of the $(m+1)$-dimensional vector spaces over the finite field $\mathbb{F}_q$.  Consider the bipartite graph $W_m(q)$ with partite sets $P$ and $L$ defined as follows:  a point $(p)=(p_1,p_2,\ldots,p_{m+1})\in P$ is adjacent to a line $[l]=[l_1,l_2,\ldots,l_{m+1}]\in L$ if and only if the following $m$ equalities hold: $l_{i+1} + p_{i+1}=l_{i}p_1$ for $i=1,\ldots, m$.  We call the graphs $W_m(q)$  Wenger graphs. In this paper, we determine all distinct eigenvalues of the adjacency matrix of $W_m(q)$ and their multiplicities. We also survey results on Wenger graphs.
\end{abstract}

\section{Introduction}

All graph theory notions can be found in Bollob\'as \cite{Bol}. Let $q=p^e$, where $p$ is a prime and $e\geq 1$ is an integer. For $m\geq 1$, let $P$ and $L$ be two copies of the $(m+1)$-dimensional vector spaces over the finite field $\mathbb{F}_q$. We call the elements of $P$ \emph{points} and the elements of $L$ \emph{lines}. If $a\in \F_q^{m+1}$, then we write $(a)\in P$ and $[a]\in L$. Consider the bipartite graph $W_m(q)$ with partite sets $P$ and $L$ defined as follows:  a point $(p)=(p_1,p_2,\ldots,p_{m+1})\in P$ is adjacent to a line $[l]=[l_1,l_2,\ldots,l_{m+1}]\in L$ if and only if the following $m$ equalities hold:
\begin{align*}
l_2+p_2&=l_1p_1\\
l_3+p_3&=l_2p_1\\
&\vdots\\
l_{m+1}+p_{m+1}&=l_m p_1.
\end{align*}
The graph $W_m(q)$ has $2q^{m+1}$ vertices, is $q$-regular and has $q^{m+2}$ edges.

In \cite{Wenger}, Wenger  introduced a family of $p$-regular bipartite graphs $H_k(p)$ as follows. For every  $k\ge 2$, and every prime $p$, the partite sets of $H_k(p)$ are two copies of integer sequences $\{0,1,\ldots, p-1\}^k$,  with vertices  $a=(a_0,a_1,\ldots,a_{k-1})$ and $b=(b_0,b_1,\ldots,b_{k-1})$  forming an edge if
$$b_j \equiv a_j + a_{j+1} b_{k-1} \pmod{p}\;\; \text{for all}\;\; j=0, \ldots, k-2.$$
The introduction and study of these graphs were motivated by an extremal graph theory problem of determining the largest number of edges in a graph of order $n$ containing no cycle of length $2k$. This parameter also known as the Tur\'{a}n number of the cycle $C_{2k}$, is denoted by $\ex(n,C_{2k})$. Bondy and Simonovits~\cite{BonSim} showed that $\ex(n,C_{2k}) = O(n^{1 + 1/k})$, $n\to \infty$. Lower bounds of magnitude $n^{1 + 1/k}$ were known (and still are) for  $k=2,3,5$ only, and the graphs $H_k(p)$, $k=2,3,5$, provided new and simpler examples of such magnitude extremal graphs. For many results on $\ex(n,C_{2k})$, see Verstra\"{e}te \cite{Vers}, Pikhurko \cite{Pikhurko} and references therein.

In \cite{LazUst}, Lazebnik and Ustimenko, using a construction based on a certain Lie algebra, arrived at a family of bipartite graphs $H'_n(q)$,  $n\ge 3$, $q$ is a prime power, whose partite sets were two copies of $\F_q^{n-1}$, with vertices $(p)=(p_2, p_3, \ldots, p_{n})$ and $[l]=[l_1,l_3,\ldots,l_{n}]$  forming an edge if
 $$l_k - p_k = l_1p_{k-1}\;\; \text{for all}\;\; k=3, \ldots, n.$$
 It is easy to see that for all $k\ge 2$ and prime $p$,  graphs $H_k(p)$ and $H'_{k+1}(p)$ are isomorphic, and the map
  \begin{align*}
\phi:  (a_0,a_1,\ldots,a_{k-1}) &\mapsto (a_{k-1}, a_{k-2}, \ldots, a_0),\\
       (b_0,b_1,\ldots,b_{k-1}) &\mapsto [b_{k-1}, b_{k-2}, \ldots, b_0],
\end{align*}
provides an isomorphism from $H_k(p)$ to $H'_{k+1}(p)$.
    Hence,  graphs $H'_n(q)$ can be viewed as generalizations of graphs $H_k(p)$.  It is also easy to show that graphs $H'_{m+2}(q)$ and $W_m(q)$ are isomorphic: the function
      \begin{align*}
\psi:  (p_2, p_3, \ldots, p_{m+2}) &\mapsto [p_2, p_3, \ldots, p_{m+2}],\\
       [l_1, l_3, \ldots, l_{m+2}] &\mapsto (-l_1, -l_3, \ldots, -l_{m+1}),
\end{align*}
mapping points to lines and lines to points, is an isomorphism of $H'_{m+2}(q)$ to $W_m(q)$. Combining this isomorphism with the results in \cite{LazUst},  we obtain that the graph $W_1(q)$ is isomorphic to an induced subgraph of the point-line incidence graph of the projective plane $PG(2,q)$,  the graph $W_2(q)$ is isomorphic to an induced subgraph of the point-line incidence graph of the generalized quadrangle $Q(4,q)$, and $W_3(q)$ is a homomorphic image of an induced subgraph of the point-line incidence graph of the generalized hexagon $H(q)$.

We call the graphs $W_m(q)$  {\it Wenger graphs}. The representation of Wenger graphs as $W_m(q)$ graphs first appeared in Lazebnik and Viglione~\cite{LazVig}. These authors suggested another useful representation of these graphs,  where the right-hand sides of
equations are represented as monomials of $p_1$ and $l_1$ only, see  \cite{VigThesis}. For this, define a bipartite graph $W'_m(q)$ with the same partite sets as $W_m(q)$,  where $(p)=(p_1, p_2, \ldots, p_{m+1})$ and $[l]=[l_1,l_2,\ldots,l_{m+1}]$ are adjacent if
\begin{equation} \label{wenger11} l_k + p_k = l_{1}p_1^{k-1}\;\; \text{for all}\;\; k=2, \ldots, m+1.
\end{equation}
The map
  \begin{align*}
\omega:  (p) &\mapsto (p_1, p_2, p_3', \ldots, p_{m+1}'),\;\text{where}\; p_k'=  p_k+ \sum_{i=2}^{k-1} p_{i}p_1^{k-i},\;k=3,\ldots, m+1,\\
       [l] &\mapsto [l_1, l_2, \ldots, l_{m+1}],
\end{align*}
defines an isomorphism from $W_m(q)$  and $W'_m(q)$.

It was shown in \cite{LazUst} that the automorphism group of $W_m(q)$ acts transitively on each of $P$ and $L$,  and on the set of edges of $W_m(q)$.   In other words,  the graphs $W_m(q)$ are point-, line-, and edge-transitive.  A more detailed study,  see \cite{LazVig},  also showed that $W_1(q)$ is vertex-transitive for all $q$, and that $W_2(q)$ is vertex-transitive for even $q$. For all $m\ge 3$ and $q\ge 3$,  and for $m=2$ and all odd $q$,  the graphs $W_m(q)$ are not vertex-transitive. Another result of \cite{LazVig} is that $W_m(q)$ is connected when $1\le m\le q-1$,  and disconnected when $m\ge q$, in which case it has $q^{m-q+1}$ components, each isomorphic to $W_{q-1}(q)$. In \cite{Viglione}, Viglione  proved that when $1\leq m\leq q-1$, the diameter of $W_m(q)$ is $2m+2$. We wish to note that the statement about the number of components of $W_{m}(q)$ becomes apparent from the representation (\ref{wenger11}).   Indeed, as $l_1p_1^i = l_1p_1^{i+q-1}$,  all points and lines in a component have the property that
 their coordinates $i$ and $j$, where $i \equiv j \mod(q-1)$, are equal.  Hence, points $(p)$,  having $p_1=\ldots = p_{q} = 0$,  and at least one distinct coordinate $p_i$, $q+1\le i\le m+1$, belong to different components. This shows that the number of components is at least $q^{m-q+1}$. As $W_{q-1}(q)$ is connected and $W_m(q)$ is edge-transitive, all components are isomorphic to $W_{q-1}(q)$. Hence, there are exactly $q^{m-q+1}$ of them. A result of Watkins \cite{Wat70}, and the edge-transitivity of $W_m(q)$ imply that the vertex connectivity (and consequently the edge connectivity) of $W_m(q)$ equals the degree of regularity $q$, for any $1\leq m\leq q-1$.

Shao, He and Shan \cite{SHS} proved that in  $W_m(q)$, $q=p^e$, $p$ prime,  for  $m\geq 2$, for any integer $l\neq 5, 4\leq l\leq 2p$ and any vertex $v$, there is a cycle of length $2l$ passing through the vertex $v$. We wish to remark that the edge-transitivity of $W_m(q)$ implies the existence of a $2l$ cycle through any edge, a stronger statement.  Li and Lih \cite{LiLih} used the Wenger graphs to determine the asymptotic behavior of the Ramsey number $r_n(C_{2k})=\Theta(n^{k/(k-1)})$ when $k\in \{2,3,5\}$ and $n\rightarrow \infty$; the Ramsey number $r_n(G)$ equals the minimum integer $N$ such that in any edge-coloring of the complete graph $K_N$ with $n$ colors, there is a monochromatic $G$.
 Representation (\ref{wenger11}) points to a relation of Wenger graphs with the moment curve $t\mapsto (1,t,t^2,t^3,..., t^m)$, and, hence,  with the Vandermonde's determinant, which was explicitly used in \cite{Wenger}. This is also in the background of some geometric constructions by Mellinger and Mubayi \cite{MelMub} of magnitude extremal graphs without short even cycles.

In Section 2, we determine the spectrum of the graphs $W_m(q)$, defined as the multiset of the eigenvalues of the adjacency matrix of $W_m(q)$. Futorny and Ustimenko \cite{FU} considered applications of Wenger graphs in cryptography and coding theory,  as well as some generalizations.  They also conjectured that the second largest eigenvalue $\lambda _2$ of the adjacency matrix of Wenger graphs $W_m(q)$ is bounded from above by $2\sqrt{q}$.
The results of this paper confirm the conjecture for $m=1$ and $2$, or $m=3$ and $q\ge 4$,  and refute it in  other cases.  We wish to point out that for $m=1$ and $2$, or $m=3$ and $q\ge 4$,  the upper bound $2\sqrt{q}$ also follows from the known values of $\lambda _2$ for the point-line $(q+1)$-regular incidence graphs of the generalized polygons $PG(2,q)$, $Q(4,q)$ and $H(q)$ and eigenvalue interlacing (see Brouwer, Cohen and Neumaier \cite{BCN}). In \cite{LiLuWa},  Li, Lu and Wang showed that the graphs $W_m(q)$,  $m=1,2$, are Ramanujan, by computing the eigenvalues of another family of graph described by systems of linear equations in \cite{LazUstDkq}, $D(k,q)$, for $k=2,3$. Their result follows from the fact that  $W_1(q)\simeq D(2,q)$,  and $W_2(q)\simeq D(3,q)$. For more on Ramanujan graphs, see Lubotzky, Phillips and Sarnak \cite{LPS},  or  Murty \cite{Murty}.  Our results also imply that for fixed $m$ and large $q$, the Wenger graph $W_m(q)$ are expanders. For more details on expanders  and their applications, see Hoory, Linial and Wigderson \cite{HLW}, and references therein.

\section{Main Results}

\begin{theorem}\label{main1}
For all prime power $q$ and $1\leq m \leq q-1$, the distinct eigenvalues of $W_m(q)$ are
\begin{equation} \label{eigen}
\pm q, \; \pm\sqrt{mq}, \;\pm\sqrt{(m-1)q}, \;\cdots ,  \;\pm\sqrt{2q}, \;\pm\sqrt{q}, \; 0 .
\end{equation}
The multiplicity of the eigenvalue $\pm\sqrt{iq}$ of $W_m(q)$,  $0\leq i \leq m$,  is
\begin{equation}\label{mult}
(q-1){q \choose i}\sum_{d=i}^m\sum_{k=0}^{d-i} (-1)^k{q-i \choose k}q^{d-i-k}.
\end{equation}
\end{theorem}
\begin{proof}
As the graph $W_m(q)$ is bipartite with partitions  $L$ and $P$, we can arrange the rows and the columns of an adjacency matrix $A$ of $W_m(q)$ such that $A$ has the following form:
 \begin{equation}A = \bordermatrix{~ & L& P \cr
                  L & 0 & N^T \cr
                  P & N & 0 \cr}
\end{equation}
which implies that
\begin{equation}A^2 = \begin{pmatrix}
                     N^TN & 0 \\
                     0 & NN^T
                    \end{pmatrix}.
\end{equation}
As the matrices $N^TN$ and $NN^T$ have the same spectrum, we just need to compute the spectrum for one of these matrices. To determine the spectrum of $N^TN$, let $H$ denote the point-graph of $W_m(q)$ on $L$. This means that the vertex set of $H$ is $L$,  and two distinct lines $[l]$ and $[l^\prime]$ of $W_m(q)$ are adjacent in $H$ if
there exists a point $(p)\in P$, such that $[l]\sim (p)\sim [l^\prime]$ in $W_m(q)$. More precisely, $[l]$ and $[l^\prime]$ are adjacent in $H$, if there exists $p_1\in \mathbb{F}_q$ such that for all $i=1,\ldots, m$,  we have
\begin{align*}
 &l_1\not = l_1^\prime \text{ and } l_{i+1}-l_{i+1}^\prime =p_1(l_i-l_i^\prime) \iff  \\
 &l_1\not = l_1^\prime \text{ and } l_{i+1}-l_{i+1}^\prime =p_1^i(l_1-l_1^\prime).
\end{align*}
This implies that $H$ is actually the Cayley graph of the additive group of the vector space $\mathbb{F}_q^{m+1}$ with a generating set
\begin{equation}\label{generatingS}
S=\{(t,tu,\ldots, tu^m) \mid t\in
\mathbb{F}_q^*,u\in \mathbb{F}_q
\}.
\end{equation}

Let $\omega$ be a complex $p$-th root of unity. For $x\in \mathbb{F}_q$, the trace of $x$ is defined as $tr(x)=\sum_{i=0}^{e-1}x^{p^i}$. The eigenvalues of $H$ are indexed after the $(m+1)$-tuples $(w_1,\dots,w_{m+1})\in \mathbb{F}_q^{m+1}$,  and can be represented in  the following form (see Babai \cite{Babai} and Lovasz \cite{Lovasz} for more details):
\begin{align*}
 \lambda_{(w_1,\ldots,w_{m+1})}&=\sum_{(t,tu,\ldots, tu^m)\in S}\omega^{tr(tw_1)}\cdot \omega^{tr(tuw_2)}\cdot \cdots  \cdot\omega^{tr(tu^mw_{m+1})}\\
                               &=\sum_{t\in \mathbb{F}_q^*,\, u\in \mathbb{F}_q}\omega^{tr(tw_1+tuw_2+ \cdots+tu^mw_{m+1})}\\
                               &=\sum_{t\in \mathbb{F}_q^*,\, u\in \mathbb{F}_q}\omega^{tr(t(f(u)))} \text{   (where $f(u):=w_1+w_2u+\dots+w_{m+1}u^m$)}\\
                               &=\sum_{t\in \mathbb{F}_q^*, \, f(u)=0}\omega^{tr(t(f(u)))}+\sum_{t\in \mathbb{F}_q^*, \,f(u)\not=0}\omega^{tr(t(f(u)))}.
\end{align*}
As $\sum_{t\in \mathbb{F}_q^*}\omega^{tr(tx)}=q-1$ for $x=0$, and $\sum_{t\in \mathbb{F}_q^*}\omega^{tr(t x)}=-1$ for every $x\in \mathbb{F}_q^*$, we obtain that
\medskip
\begin{equation}\label{eigenvalue}
 \lambda_{(w_1,\ldots,w_{m+1})}=|\{u\in \mathbb{F}_q \mid f(u)=0\}| (q-1)-|\{u\in \mathbb{F}_q \mid f(u)\not = 0\}|.
\end{equation}
\smallskip

Let $B$ be the adjacency matrix of $H$. Then $N^TN=B+qI$; this fact can be seen easily by examining the on- and off-diagonal entries of both sides of the equation. Therefore, the eigenvalues of $W_m(q)$ can be written in the form $$\pm \sqrt{\lambda_{(w_1,\ldots,w_{m+1})}+q},$$
where $(w_1,\ldots,w_{m+1})\in \mathbb{F}_q^{m+1}$.  Let $f(X)= w_1 + w_2X + \cdots + w_{m+1}X^{m}\in \F _q[X]$.   We consider two cases.
\begin{enumerate}
\item $f=0$.  In this case, $|\{u\in \mathbb{F}_q \mid f(u)=0\}|=q$, and $\lambda_{(w_1,\ldots,w_{m+1})}=q(q-1)$. Thus, $W_m(q)$ has $\pm q$ as its eigenvalues.
\item $f\not = 0$. In this case, let $i=|\{u\in \mathbb{F}_q \mid f(u)=0\}|\leq m$ as $1\leq m\leq q-1$. This shows that $\lambda_{(w_1,\dots,w_{m+1})}=i(q-1)-(q-i)=iq-q$ and implies that $\pm \sqrt{\lambda_{(w_1,\dots,w_{m+1})}+q}=\pm \sqrt{iq}$ are eigenvalues of $W_m(q)$. Note that for any $0\leq i\leq m$, there exists a polynomial $f$ over $\F_q$ of degree at most $m\leq q-1$, which has exactly $i$ distinct roots in $\F _q$.  For such $f$,  $|\{u\in \mathbb{F}_q \mid f(u)=0\}|=i$, and,
    hence, there exists $(w_1,\ldots,w_{m+1})\in \mathbb{F}_q^{m+1}$, such that $\lambda_{(w_1,\ldots,w_{m+1})}=iq-q$. Thus, $W_m(q)$ has $\pm \sqrt{iq}$ as its eigenvalues, for any $0\leq i\leq m$,  and the first statement of the theorem is proven.
\end{enumerate}

The arguments above  imply that the multiplicity of the eigenvalue $\pm\sqrt{iq}$ of $W_m(q)$ equals the number of polynomials of
degree at most $m$ (not necessarily monic) having exactly $i$ distinct roots in $\F _q$.
To calculate these multiplicities, we need the following lemma.  Particular cases of the lemma  were considered in  Zsigmondy \cite{Zsig}, and in Cohen \cite{Cohen73}.  The complete result appears in A.~Knopfmacher and J.~Knopfmacher \cite{KoKo}.

\begin{lemma}[\cite{KoKo}] \label{prop2}
Let $q$ be a prime power, and let $d$ and $i$ be integers such that $0\leq i\leq d\le q-1$.
Then the number $b(q,d,i)$ of monic polynomials in $\F_q[X]$ of degree $d$, having exactly $i$ distinct roots in $\F _q$ is given by
\begin{equation}
b(q,d,i) ={q \choose i}\sum_{k=0}^{d-i} (-1)^k{q-i \choose k}q^{d-i-k}.
\end{equation}
\end{lemma}

By Lemma \ref{prop2}, the number  of polynomials of
degree at most $m$ in $\F_q[X]$ (not necessarily monic) having exactly $i$ distinct roots in $\F _q$ is
\begin{equation}
\sum_{ d=i}^m(q-1)\,b(q,d,i) =(q-1){q \choose i}\sum_{d=i}^m\sum_{k=0}^{d-i} (-1)^k{q-i \choose k}q^{d-i-k}.
\end{equation}
This concludes the proof the theorem.
\end{proof}

The previous result shows that $W_m(q)$ is connected and has $2m+3$ distinct eigenvalues, for any $1\leq m\leq q-1$. As the diameter of a graph is strictly less than the number of distinct eigenvalues (see \cite[Section 4.1]{BCN} for example), this implies that the diameter of Wenger graph is less or equal to $2m+2$. This is actually the exact value of the diameter of the Wenger graph as shown by Viglione \cite{Viglione}.

Since the sum of multiplicities of all eigenvalues of the graph $W_m(q)$ is equal to its order,  and remembering
that the multiplicity of $\pm q$ is one when $1\le m\le q-1$,  we have a combinatorial proof of the following identity.

\begin{corollary}  For every prime power $q$,  and every $m$, $1\le m \le q-1$,
\begin{equation}\label{identity}
\sum_{i=0}^m {q \choose i}\sum_{d=i}^m\sum_{k=0}^{d-i} (-1)^k{q-i \choose k}q^{d-i-k}= \frac{q^{m+1}-1}{q-1}.
\end{equation}
\end{corollary}

The identity (\ref{identity}) seems to hold for all  integers  $q\ge 3$,  so a direct proof is desirable. Other identities can be obtained by taking the higher moments of the eigenvalues of $W_m(q)$.
\bigskip

As we discussed in the introduction,  for $m\ge q$, the graph $W_m(q)$ has $q^{m-q+1}$ components, each isomorphic to $W_{q-1}(q)$.
This, together with Theorem \ref{main1}, immediately implies the following.
\begin{proposition}\label{main2}
For $ m \geq q$, the distinct eigenvalues of $W_m(q)$ are
$$\pm q, \;\pm\sqrt{(q-1)q},\; \pm\sqrt{(q-2)q},\; \cdots , \; \pm\sqrt{2q}, \;\pm\sqrt{q}, \;0 ,$$
and the multiplicity of the eigenvalue $\pm\sqrt{iq}$, $0\le i\le q-1$,  is
$$(q-1)q^{m+1-q} {q \choose i}\sum_{d=i}^q\sum_{k=0}^{d-i} (-1)^k{q-i \choose k}q^{d-i-k}.$$
\end{proposition}

\section{Open Questions}

There are several open questions about the Wenger graphs $W_m(q)$ that we think are worth investigating: deciding whether these graphs are Hamiltonian, finding the lengths of all their cycles,  determining their automorphism group, or determining the parameters of the linear codes whose Tanner graphs are the Wenger graphs.

\section*{Acknowledgments}
The research of S.M. Cioab\u{a} was partially supported by National Security Agency grant H98230-13-1-0267,  and the research of F. Lazebnik was partially supported by the NSF grant DMS-1106938-002.

\end{document}